\documentclass[11pt]{article}

\usepackage[utf8]{inputenc}
\usepackage[T1]{fontenc}
\usepackage{lmodern}
\usepackage{microtype}
\usepackage{csquotes}

\usepackage[a4paper,margin=1in]{geometry}
\usepackage{setspace}
\usepackage{float}

\usepackage{amsmath,amssymb,amsfonts,amsthm}
\usepackage{mathtools}
\usepackage{bm}
\usepackage{dsfont}
\numberwithin{equation}{section}
\allowdisplaybreaks

\usepackage{graphicx}
\usepackage{booktabs}
\usepackage{array}
\usepackage{multirow}
\usepackage{xcolor}
\usepackage{subcaption}
\usepackage[backend=biber,style=alphabetic,maxnames=8,giveninits=true]{biblatex}
\usepackage{aliascnt}

\usepackage{hyperref}
\hypersetup{
  colorlinks=true,
  linkcolor=[rgb]{0.1,0.1,0.6},
  citecolor=[rgb]{0.0,0.45,0.0},
  urlcolor=[rgb]{0.0,0.3,0.6},
  pdfauthor={Witold Płecha},
  pdftitle={The Distributional Tail of Worst--Case Quickselect}
}
\usepackage[nameinlink,capitalise]{cleveref}

\newtheorem{theorem}{Theorem}[section]

\newaliascnt{lemma}{theorem}
\newtheorem{lemma}[lemma]{Lemma}
\aliascntresetthe{lemma}

\newaliascnt{proposition}{theorem}
\newtheorem{proposition}[proposition]{Proposition}
\aliascntresetthe{proposition}

\newaliascnt{corollary}{theorem}
\newtheorem{corollary}[corollary]{Corollary}
\aliascntresetthe{corollary}

\newaliascnt{claim}{theorem}

\aliascntresetthe{claim}

\theoremstyle{definition}

\newaliascnt{definition}{theorem}

\aliascntresetthe{definition}

\newaliascnt{assumption}{theorem}

\aliascntresetthe{assumption}

\newaliascnt{example}{theorem}

\aliascntresetthe{example}

\theoremstyle{remark}

\newaliascnt{remark}{theorem}
\newtheorem{remark}[remark]{Remark}
\aliascntresetthe{remark}

\crefname{theorem}{theorem}{theorems}
\Crefname{theorem}{Theorem}{Theorems}

\crefname{lemma}{lemma}{lemmas}
\Crefname{lemma}{Lemma}{Lemmas}

\crefname{proposition}{proposition}{propositions}
\Crefname{proposition}{Proposition}{Propositions}

\crefname{corollary}{corollary}{corollaries}
\Crefname{corollary}{Corollary}{Corollaries}

\crefname{claim}{claim}{claims}
\Crefname{claim}{Claim}{Claims}

\crefname{definition}{definition}{definitions}
\Crefname{definition}{Definition}{Definitions}

\crefname{assumption}{assumption}{assumptions}
\Crefname{assumption}{Assumption}{Assumptions}

\crefname{example}{example}{examples}
\Crefname{example}{Example}{Examples}

\crefname{remark}{remark}{remarks}
\Crefname{remark}{Remark}{Remarks}

\crefname{section}{section}{sections}
\Crefname{section}{Section}{Sections}

\crefname{equation}{equation}{equations}
\Crefname{equation}{Equation}{Equations}

\crefname{appendix}{appendix}{appendices}
\Crefname{appendix}{Appendix}{Appendices}

\usepackage{enumitem}
\setlist{nosep}

\DeclareMathOperator{\supp}{supp}

\newcommand{\R}{\mathbb{R}}

\newcommand{\PP}[1]{\mathbb{P}\!\left(#1\right)}
\newcommand{\EE}[1]{\mathbb{E}\!\left[#1\right]}

\newcommand{\1}{\mathds{1}}

\newcommand{\dd}{\mathrm{d}}
\newcommand{\defeq}{\vcentcolon=}

\newcommand{\eqd}{\stackrel{d}{=}}

\newcommand{\bigO}{\mathcal{O}}
\newcommand{\smallo}{o}

\newcommand{\Unif}{\mathrm{Unif}}
\newcommand{\Gam}{\Gamma}
\newcommand{\Exp}{\mathrm{Exp}}

\newcommand{\Svar}{S}             
\newcommand{\Tn}{T_n}             
\newcommand{\U}{U}                

\newcommand{\aj}{a_j(t)}
\newcommand{\athr}[1]{\log\!\frac{#1+1}{t}} 

\title{The Distributional Tail of Worst--Case Quickselect}

\author{
  Witold Płecha\\
  \small Mathematical Institute, University of Wrocław
}
\date{\today}

\begin{document}

\maketitle

\begin{abstract}
We study the almost surely finite random variable $S$ defined by the distributional fixed-point equation
\[
S \eqd 1+\max\{US',(1-U)S''\},
\qquad U\sim \Unif(0,1),
\]
where $S'$ and $S''$ are independent copies of $S$, independent of $U$.
This random variable arises as the almost sure limit of the normalized worst-case number of key comparisons used by classical \textsc{Quickselect} with uniformly chosen pivots in the model of Devroye.

Our first contribution concerns the right tail of $S$.
We prove explicit one-sided bounds for the rate function $-\log \PP{S>t}$ and, in particular, identify its first-order asymptotic growth:
\[
-\log \PP{S>t} = t\log t + \bigO\!\bigl(t\log\log t\bigr),
\qquad t\to\infty.
\]
The argument combines a binary-search-tree embedding and a one-level second-moment method with a moment-generating-function comparison inspired by ideas of Alsmeyer and Dyszewski for the nonhomogeneous smoothing transform.
As a byproduct, we obtain an explicit pointwise Chernoff majorant for the tail.

Our second contribution is a distribution-function scheme for deriving explicit upper bounds on $\EE{S}$.
Starting from the fixed-point equation at the level of the distribution function, we construct an order-preserving lower iteration and a conservative mesh discretization suited to computer-assisted upper bounds on the mean.
We illustrate the latter numerically in floating-point arithmetic, but do not pursue a certified numerical proof here.
\end{abstract}

\medskip
\noindent\textbf{Acknowledgements.}
This work grew out of the author's master's thesis prepared under the supervision of Piotr Dyszewski.
I thank him for suggesting the problem, for his guidance, and for many helpful discussions.

\section{Introduction and main results}

Quickselect, introduced by Hoare under the name \emph{Find}, is one of the classical randomized algorithms for selecting an order statistic from a finite set of distinct keys \cite{hoare1961find}. Its analysis forms part of the broader early tradition of the mathematical study of recursive algorithms; see, for example, \cite{knuth1972mathematical}. A substantial probabilistic literature on Quickselect emerged later.

Among the most studied cost measures is the number of key comparisons. In the classical unit-cost model, Quickselect has been analyzed probabilistically in a substantial body of work; see, for example, \cite{gruebel-roesler-1996,gruebel1998markov,hwang-tsai-2002,devroye2001find}. At the same time, probabilistic analyses of Quickselect are not confined to the unit-cost setting: more general comparison costs, including symbol-comparison models, were studied in \cite{fill-nakama-2013,fill-matterer-2014,vallee-clement-fill-flajolet-2009}. The classical key-comparison model may therefore be viewed as a special case within a broader probabilistic framework for recursive selection algorithms.

Within the unit-cost model, these works address several rather different aspects of the key-comparison complexity. Grübel and Rösler \cite{gruebel-roesler-1996} develop an asymptotic distribution theory for normalized costs across rank regimes, while Grübel \cite{gruebel1998markov} gives a complementary structural treatment from a Markov-chain perspective. Hwang and Tsai \cite{hwang-tsai-2002} analyze a different limiting regime, involving small target ranks and the appearance of the Dickman function. Devroye \cite{devroye2001find}, in turn, studies the cost in the worst case over the choice of the target rank, which is the perspective relevant for the present paper.

For this worst-case cost over the target rank, the natural starting point is the random binary-search-tree representation due to Devroye~\cite{devroye2001find}. In this framework, if $\Tn$ denotes the worst-case number of key comparisons used by \textsc{Quickselect} on $n$ distinct keys, then
\[
\frac{\Tn}{n} \to \Svar
\qquad\text{almost surely},
\]
where the limit random variable $\Svar$ is characterized by the fixed-point equation
\[
\Svar \eqd 1+\max\{\U \Svar',(1-\U)\Svar''\},
\qquad \U\sim\Unif(0,1),
\]
with $\Svar'$ and $\Svar''$ independent copies of $\Svar$, independent of $\U$.
Grübel--Rösler~\cite{gruebel-roesler-1996} identify the limit law through this fixed-point equation and establish uniqueness of the proper solution. In Devroye's formulation~\cite{devroye2001find}, the same limit variable is shown to be supported on $[2,\infty)$, to admit a bounded Lipschitz density, and to have finite moments of all orders.

Devroye~\cite{devroye2001find} also obtained explicit, though rather coarse, quantitative bounds on the law of $\Svar$. In particular,
\[
2+2\log 2 \le \EE{\Svar} \le 1+\frac{5}{\sqrt{2\pi}}+\frac{12e}{5} < 9.5185,
\]
and his argument yields an explicit exponential upper bound on the tail. In terms of the rate function, however, this gives only a linear lower bound on $-\log \PP{\Svar>t}$. More precisely,
\[
\PP{\Svar\ge t}
\le
\frac{\EE{\Svar-1}\,(t-1)e^{-(t-4)/3}}{9},
\qquad t\ge 4,
\]
and hence
\[
-\log \PP{\Svar>t}\ge \frac{t}{3}-\bigO(\log t).
\]
More generally, Devroye proved that the tail is superexponentially small, but without identifying the first-order logarithmic growth of its rate. Related exponential-type tail bounds for broader classes of max-recursive equations were later obtained by R\"uschendorf and Schopp~\cite{ruschendorf-schopp-2006}, but these general results likewise do not determine the sharp first-order asymptotic form in the present Quickselect setting.

The main goal of the present paper is to identify this first-order growth. To avoid repeatedly switching between statements about the tail probability itself and its logarithm, we introduce the rate function
\[
I(t):=-\log \PP{\Svar>t},
\qquad t>0.
\]
Our first result gives explicit one-sided bounds on $I(t)$.

\begin{proposition}[Explicit bounds for the rate function]
\label{prop:intro-rate-bounds}
For every $t>2$,
\[
I(t)\ge t\log t-(1+\log 2)t.
\]
Moreover, as $t\to\infty$,
\[
I(t)\le t\log t+t\log\log t-t\log 2+\smallo(t).
\]
\end{proposition}

The lower bound comes from a pointwise Chernoff estimate obtained through a comparison principle for the moment generating function.
This part of the argument is inspired by the work of Alsmeyer and Dyszewski on the nonhomogeneous smoothing transform~\cite{alsmeyer-dyszewski-2017}.
The upper bound is obtained by a one-level second-moment method on the binary-search-tree side: we count vertices on a carefully chosen level whose multiplicative weights already force the path sum to exceed the threshold.

Combining the two sides gives our main asymptotic theorem.

\begin{theorem}[Main tail asymptotic]
\label{thm:intro-main-tail}
As $t\to\infty$,
\[
I(t)=t\log t+\bigO\!\bigl(t\log\log t\bigr).
\]
Equivalently,
\[
-\log \PP{\Svar>t}=t\log t+\bigO\!\bigl(t\log\log t\bigr).
\]
\end{theorem}

Our second aim is to develop a practical framework for improving upper bounds on the mean $\EE{\Svar}$.
Rather than working directly with the random fixed-point equation, we pass to the distribution function
\[
F(x):=\PP{\Svar\le x}.
\]
The fixed-point identity for $\Svar$ then becomes a nonlinear integral equation for $F$, which naturally suggests an iterative lower scheme for the distribution function.
When combined with the pointwise tail majorant underlying \Cref{prop:intro-rate-bounds}, this leads to explicit upper bounds on $\EE{\Svar}$ through the tail-integral representation
\[
\EE{\Svar}=2+\int_2^\infty (1-F(x))\,\dd x.
\]

More precisely, we construct an order-preserving lower iteration for $F$ and a conservative mesh discretization of the associated nonlinear operator.
This produces a framework for computer-assisted upper bounds on $\EE{\Svar}$.
In the present paper we include only a floating-point pilot computation, intended as an illustration of the method rather than a certified numerical proof.

The paper is organized as follows.
\Cref{sec:tree} recalls the binary-search-tree embedding of \textsc{Quickselect}, introduces the limit functional, and derives the fixed-point equation for $\Svar$.
\Cref{sec:upper} proves the upper bound in \Cref{prop:intro-rate-bounds} by a one-level second-moment argument.
\Cref{sec:lower} develops the moment-generating-function comparison, derives the pointwise Chernoff bound, and proves the lower bound in \Cref{prop:intro-rate-bounds}.
\Cref{sec:mean} turns to the mean, develops the distribution-function iteration, and describes the conservative discretization scheme together with the floating-point pilot computation.
For convenience, some standard auxiliary material is collected in \Cref{app:aux}.

\section{Tree embedding and worst--case cost}
\label{sec:tree}

In this section we recall the binary-search-tree representation of the worst--case cost of
\textsc{Quickselect}, following Devroye~\cite{devroye2001find}, and use it to identify the
limiting fixed-point variable.

\subsection{Quickselect and the recursion for comparisons}

Given \(n\) distinct keys and a target rank \(k \in \{1,\dots,n\}\), \textsc{Quickselect} chooses a
pivot uniformly at random, compares each of the remaining \(n-1\) keys with it, and partitions
the input into the two subarrays consisting of the keys smaller and larger than the pivot.
If the pivot has rank \(k\), the algorithm stops. If its rank is \(J>k\), the algorithm recurses on
the left subarray of size \(J-1\); if \(J<k\), it recurses on the right subarray of size \(n-J\), now
seeking the \((k-J)\)-th smallest element there.

Let \(T_{n,k}\) denote the number of key comparisons used to select the \(k\)-th smallest element
among \(n\) distinct keys. Conditioning on the pivot rank \(J\), which is uniform on
\(\{1,\dots,n\}\), yields the distributional recursion
\[
T_{n,k}
\eqd
n-1+
\begin{cases}
T_{J-1,k}, & k<J,\\
T_{n-J,k-J}, & k>J,\\
0, & k=J.
\end{cases}
\]
The worst--case cost over all target ranks is then
\[
T_n:=\max_{1\le k\le n} T_{n,k},
\]
that is, the largest number of comparisons taken over all possible target ranks.

\subsection{Embedding into an infinite binary tree}
\label{subsec:bst-embedding}

A convenient way to analyze \(T_n\) is to realize the recursion on a single infinite probability
space; compare Devroye~\cite{devroye2001find}.
Since \textsc{Quickselect} depends only on the relative order of the input keys, we may assume
without loss of generality that the input is \(\{1,\dots,n\}\).

We work on the infinite full binary tree \(\mathcal T\), and attach to each vertex \(v\in\mathcal T\)
an independent random label
\[
\xi_v\sim\Unif(0,1).
\]
If the current subproblem size at \(v\) is \(m\ge 1\), we define the pivot rank by
\[
J_{m,v}:=1+\lfloor m\xi_v\rfloor.
\]
Then \(J_{m,v}\) is uniform on \(\{1,\dots,m\}\), and the left and right child subproblem sizes are
\[
J_{m,v}-1=\lfloor m\xi_v\rfloor,
\qquad
m-J_{m,v}=\lfloor m(1-\xi_v)\rfloor
\qquad\text{a.s.}
\]
Once a subproblem size becomes \(0\), all of its descendants also have size \(0\).
Thus the worst-case cost may be represented as
\[
T_n=\sup_p \sum_{r\ge 0}(N_r(p)-1)_+,
\]
where the supremum is taken over all infinite rooted paths \(p\) in \(\mathcal T\), and
\(N_r(p)\) denotes the subproblem size at depth \(r\) along \(p\), with \(N_0(p)=n\).

Now fix an infinite rooted path \(p=(v_0,v_1,v_2,\dots)\) in \(\mathcal T\), where \(v_0\) is the root.
For each \(r\ge 1\), define the path factor \(U_r(p)\) by
\[
U_r(p)=
\begin{cases}
\xi_{v_{r-1}}, & \text{if } p \text{ goes from } v_{r-1} \text{ to its left child},\\[1mm]
1-\xi_{v_{r-1}}, & \text{if } p \text{ goes from } v_{r-1} \text{ to its right child}.
\end{cases}
\]
We also define the cumulative path weights
\[
W_r(p):=\prod_{i=1}^r U_i(p),
\qquad W_0(p):=1.
\]
Then, for each fixed path \(p\), the random variables \(U_1(p),U_2(p),\dots\) are i.i.d.
\(\Unif(0,1)\), and the selected subproblem sizes satisfy
\[
N_{r+1}(p)=\lfloor U_{r+1}(p)\,N_r(p)\rfloor,
\qquad r\ge 0.
\]
This immediately yields multiplicative upper and lower bounds.

\begin{figure}[t]
    \centering
    \includegraphics[width=1.0\linewidth]{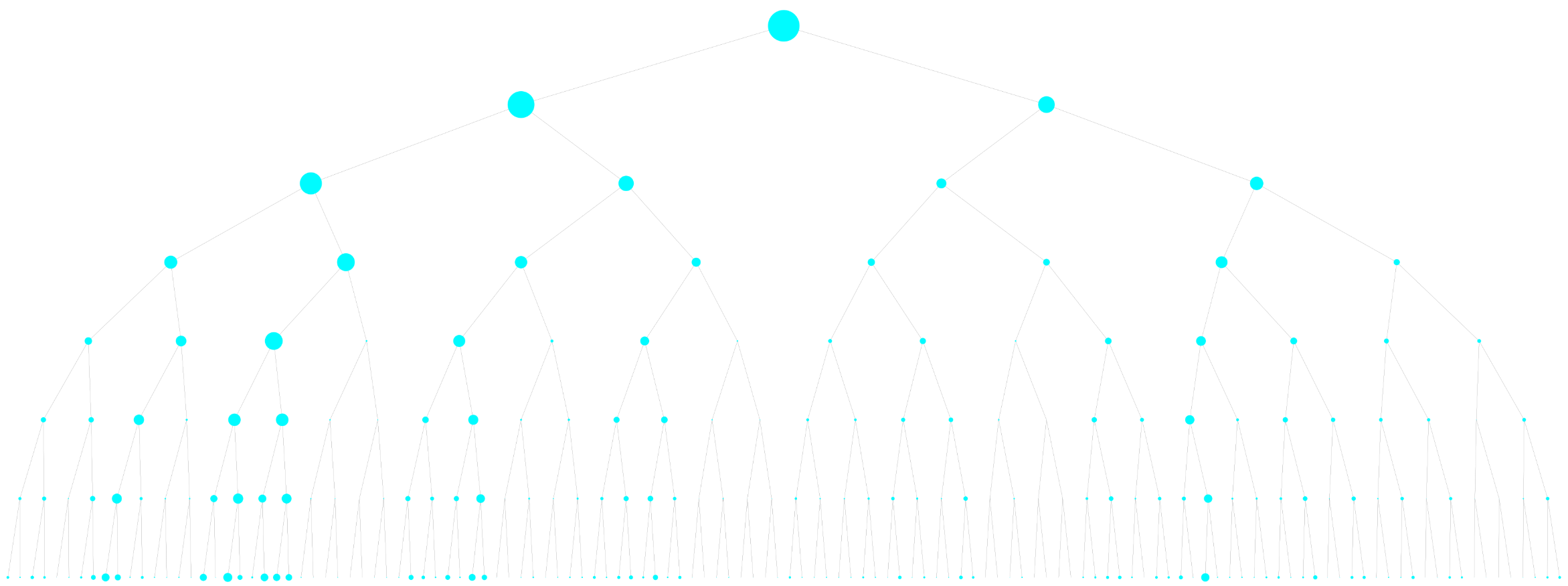}
    \caption{A random realization of the labelled binary tree from \Cref{subsec:bst-embedding}, shown up to depth \(7\).
    For each vertex \(v\) at depth \(r\), the area of the disk is proportional to \(W_r(p)\) for any infinite rooted path \(p\) satisfying \(p[r]=v\).
    Since \(W_r(p)\) depends only on the first \(r\) steps, this quantity is independent of the choice of the continuation of \(p\).}
\end{figure}

\begin{lemma}[Multiplicative floor bounds]
\label{lem:mult-floor-bounds}
Fix an infinite rooted path \(p\).
Then, for every \(r\ge 0\),
\[
n\,W_r(p)-r
\;\le\;
N_r(p)
\;\le\;
n\,W_r(p).
\]
\end{lemma}

\begin{proof}
By construction,
\[
N_{r+1}(p)=\lfloor U_{r+1}(p)\,N_r(p)\rfloor,
\qquad r\ge 0.
\]
The upper bound follows immediately from
\[
N_{r+1}(p)\le U_{r+1}(p)\,N_r(p),
\]
and an induction on \(r\).

For the lower bound, we use
\[
N_{r+1}(p)=\lfloor U_{r+1}(p)\,N_r(p)\rfloor
\ge U_{r+1}(p)\,N_r(p)-1.
\]
We argue by induction on \(r\).
For \(r=0\), the claim is trivial since \(N_0(p)=n=n\,W_0(p)\).
Assume that
\[
N_r(p)\ge n\,W_r(p)-r.
\]
Then
\[
N_{r+1}(p)
\ge
U_{r+1}(p)\bigl(n\,W_r(p)-r\bigr)-1
=
n\,W_{r+1}(p)-r\,U_{r+1}(p)-1
\ge
n\,W_{r+1}(p)-(r+1),
\]
since \(U_{r+1}(p)\le 1\).
This proves the lower bound.
\end{proof}

Thus, along any fixed path, the subproblem sizes decrease, up to additive errors, roughly
multiplicatively with i.i.d. \(\Unif(0,1)\) factors.
The advantage of the present construction is that all paths now live on a single probability space,
which will allow us in the next subsection to define the limit functional by taking a supremum
over all infinite rooted paths simultaneously.

\subsection{Limiting functional and convergence of \(T_n/n\)}

Recall that
\[
T_n=\sup_p \sum_{r\ge 0} (N_r(p)-1)_+,
\]
where the supremum is taken over all infinite rooted paths \(p\) in \(\mathcal T\), and \(N_r(p)\)
denotes the subproblem size at depth \(r\) along \(p\).
This is equivalent to the usual maximum over the finite recursion tree, because once \(N_r(p)=0\),
all subsequent terms vanish.

For each path \(p\), let \(W_r(p)\) be the cumulative path weights introduced in
\Cref{subsec:bst-embedding}.
Then, by \Cref{lem:mult-floor-bounds}, for every \(r\ge 0\),
\[
n\,W_r(p)-r \le N_r(p) \le n\,W_r(p).
\]

\paragraph{Pathwise approximation.}
Summing these inequalities from \(r=0\) to \(m\) gives
\begin{equation}\label{eq:pathwise-sandwich}
n\sum_{r=0}^m W_r(p)-\sum_{r=0}^m r
\;\le\;
\sum_{r=0}^m N_r(p)
\;\le\;
n\sum_{r=0}^m W_r(p)
\;\le\;
n\sum_{r=0}^\infty W_r(p).
\end{equation}
Since
\[
\sum_{r=0}^m r=\frac{m(m+1)}{2},
\]
it follows that whenever \(m=m(n)\) satisfies
\[
\frac{m(n)^2}{n}\to 0
\qquad (n\to\infty),
\]
we have, for each fixed path \(p\),
\[
\frac1n \sum_{r=0}^{m(n)} N_r(p)
\longrightarrow
\sum_{r\ge 0} W_r(p)
\qquad\text{a.s.}
\]

\paragraph{Truncated functionals and the candidate limit.}
Because all paths are realized on the common probability space introduced in
\Cref{subsec:bst-embedding}, we may define, for \(m\ge 0\),
\[
S^{(m)}:=\sup_p \sum_{r=0}^m W_r(p),
\]
and
\begin{equation}\label{eq:def-S}
S:=\sup_p \sum_{r=0}^\infty W_r(p),
\end{equation}
where the suprema are taken over all infinite rooted paths \(p\) in \(\mathcal T\).

For each fixed path \(p\), the partial sums \(\sum_{r=0}^m W_r(p)\) increase to
\(\sum_{r=0}^\infty W_r(p)\). Hence
\[
S
=
\sup_p \sum_{r=0}^\infty W_r(p)
=
\sup_p \sup_{m\ge 0} \sum_{r=0}^m W_r(p)
=
\sup_{m\ge 0} S^{(m)},
\]
so that
\[
S^{(m)} \uparrow S
\qquad\text{a.s.}
\]

\paragraph{From pathwise bounds to \(T_n\).}
Because \((N_r(p)-1)_+ \ge N_r(p)-1\), for every \(m\ge 0\) and every path \(p\),
\[
\sum_{r=0}^m (N_r(p)-1)_+
\;\ge\;
\sum_{r=0}^m N_r(p)-(m+1).
\]
Taking the supremum over all paths and using \Cref{eq:pathwise-sandwich}, we obtain
\[
T_n
\ge
\sup_p \Bigl(\sum_{r=0}^m N_r(p)-(m+1)\Bigr)
\ge
\sup_p \Bigl(n\sum_{r=0}^m W_r(p)-\frac{m(m+1)}2-(m+1)\Bigr),
\]
hence
\begin{equation}\label{eq:Tn-lower-bound}
T_n
\;\ge\;
n S^{(m)}-\frac{(m+1)(m+2)}2.
\end{equation}
On the other hand, since \((N_r(p)-1)_+ \le N_r(p)\) and \(N_r(p)\le nW_r(p)\),
\begin{equation}\label{eq:Tn-upper-bound}
T_n
\;\le\;
\sup_p \sum_{r\ge 0} N_r(p)
\;\le\;
n S.
\end{equation}

Dividing \Cref{eq:Tn-lower-bound,eq:Tn-upper-bound} by \(n\), and choosing \(m=m(n)\) so that
\(m(n)\to\infty\) and \(m(n)^2/n\to 0\), we obtain
\[
S^{(m(n))}-o(1)\le \frac{T_n}{n}\le S.
\]
Since \(S^{(m)}\uparrow S\) almost surely, it follows that \(S^{(m(n))}\to S\) almost surely.
Therefore
\[
\frac{T_n}{n}\to S
\qquad\text{almost surely}.
\]

\paragraph{Consistency with the fixed point.}
Starting from \Cref{eq:def-S}, let \(\varnothing\) denote the root of \(\mathcal T\), and set
\[
U:=\xi_{\varnothing}\sim \Unif(0,1).
\]
Every infinite rooted path in \(\mathcal T\) begins by moving either to the left or to the right
child of the root.
If a path enters the left subtree, then all subsequent cumulative weights are multiplied by \(U\);
if it enters the right subtree, they are multiplied by \(1-U\).
Therefore
\[
S
=
1+\max\{U S^{(L)},(1-U)S^{(R)}\}
\qquad\text{a.s.},
\]
where \(S^{(L)}\) and \(S^{(R)}\) are the analogues of \(S\) constructed from the left and right
subtrees, respectively.

Since the vertex labels on the two subtrees are independent, and each subtree has the same law
as the whole infinite full binary tree, the random variables \(S^{(L)}\) and \(S^{(R)}\) are
independent copies of \(S\), and they are independent of \(U\).
Consequently,
\[
S \eqd 1+\max\{U S', (1-U) S''\},
\]
where \(S',S'' \eqd S\) are independent and independent of \(U\).
This is precisely the fixed-point equation announced above.

\subsection{Finiteness and uniqueness of the fixed point}

We have now defined \(S\) by \Cref{eq:def-S}, shown that
\[
\frac{T_n}{n}\to S
\qquad\text{almost surely},
\]
and verified that \(S\) satisfies the distributional fixed-point equation
\[
S \eqd 1+\max\{US',(1-U)S''\}.
\]
It remains to recall that \(S\) is indeed almost surely finite and that this fixed-point equation
has a unique proper solution.

These facts are known from the literature.
Devroye~\cite{devroye2001find}, building on work of Grübel and Rösler~\cite{gruebel-roesler-1996},
proved that \(S\) is a proper random variable: \(S<\infty\) almost surely, \(S\) has a bounded
Lipschitz density, and
\[
\supp(S)\subseteq [2,\infty).
\]
In fact, \(S\) has finite moments of all orders, and Devroye also obtained explicit exponential and
superexponential upper bounds on the tail.

Furthermore, Grübel and Rösler~\cite{gruebel-roesler-1996} proved that among all nonnegative
random variables, the distributional identity
\[
X \eqd 1+\max\{UX',(1-U)X''\}
\]
admits exactly one proper solution.
Consequently, the random variable \(S\) constructed above is the unique proper solution of the fixed-point equation.

\section{Upper bound: level counts and the second-moment method}
\label{sec:upper}

In this section we prove the upper bound in \Cref{prop:intro-rate-bounds}.

\begin{theorem}[Upper tail via a one-level second-moment method]
\label{thm:upper-goal}
As \(t\to\infty\),
\[
-\log \PP{\Svar>t}
\le
t\log t+t\log\log t-t\log 2+\smallo(t).
\]
\end{theorem}

\noindent\emph{Idea of the proof.}
Fix a depth \(j\), and let \(v\) be a vertex at depth \(j\).
Choose any infinite rooted path \(p\) such that \(p[j]=v\), and define
\[
W(v)\defeq W_j(p).
\]
This is well defined, since \(W_j(p)\) depends only on the first \(j\) steps of \(p\), that is,
only on the unique path from the root to \(v\).

If \(W(v)\ge t/(j+1)\), then for any infinite rooted path \(p\) passing through \(v\),
\[
\Svar
\ge
\sum_{q=0}^{j} W_q(p)
\ge
(j+1)\,W(v)
\ge
t.
\]
We therefore count, at level \(j\), the number \(Z_j(t)\) of vertices satisfying this condition.
We estimate the first and second moments of \(Z_j(t)\), apply the second-moment method to
show that \(\PP{Z_j(t)\ge 1}\) is asymptotically comparable to \(\EE{Z_j(t)}\) in the relevant range
of \(j\), and then optimize over \(j\) to conclude \Cref{thm:upper-goal}.

\subsection{Level-wise setup and a sufficient condition}

Fix \(j\ge 0\), and let \(\mathcal L_j\) be the set of vertices at depth \(j\) in \(T\).
For \(v\in\mathcal L_j\), let \(U_1(v),\dots,U_j(v)\) denote the successive path factors along the
unique path from the root to \(v\).
Equivalently, if \(p\) is any infinite rooted path with \(p[j]=v\), then
\[
U_i(v)=U_i(p),\qquad 1\le i\le j,
\]
and this is well defined because only the first \(j\) steps of \(p\) matter.

For \(1\le i\le j\), define
\[
X_i(v)\defeq -\log U_i(v),
\]
and for \(0\le q\le j\), set
\[
S_q(v)\defeq \sum_{i=1}^q X_i(v),
\qquad S_0(v)\equiv 0.
\]
Then \(X_1(v),\dots,X_j(v)\) are i.i.d. \(\Exp(1)\), and hence, for \(j\ge 1\),
\[
S_j(v)\sim \Gam(j,1).
\]
Moreover,
\[
e^{-S_q(v)}=\prod_{i=1}^q U_i(v),
\qquad 0\le q\le j.
\]
Since \(S_q(v)\) is nondecreasing in \(q\),
\[
\sum_{q=0}^{j} e^{-S_q(v)}
\ge
(j+1)e^{-S_j(v)}.
\]

Fix \(t>0\), and for each integer \(j\ge 0\), set
\[
\aj \defeq \athr{j}.
\]
Hence, for any \(v\in\mathcal L_j\),
\[
S_j(v)\le \aj
\quad\Longrightarrow\quad
e^{-S_j(v)}\ge \frac{t}{j+1}
\quad\Longrightarrow\quad
\sum_{q=0}^{j} e^{-S_q(v)}\ge t.
\]

Define the level count
\[
Z_j(t)\defeq \sum_{v\in\mathcal L_j} \1\{S_j(v)\le \aj\}.
\]
If \(Z_j(t)\ge 1\), then there exists \(v\in\mathcal L_j\) with \(S_j(v)\le \aj\).
Choosing any infinite rooted path \(p\) with \(p[j]=v\), we obtain
\[
\Svar
\ge
\sum_{q=0}^{j} W_q(p)
=
\sum_{q=0}^{j} e^{-S_q(v)}
\ge
t.
\]
Thus we have the level-wise sufficient condition
\begin{equation}\label{eq:suff-cond}
\{Z_j(t)\ge 1\}\subseteq \{\Svar>t\}
\qquad\text{for every }j\ge 0.
\end{equation}

\subsection{Expectation on a level}

By linearity of expectation and the fact that every vertex \(v\in\mathcal L_j\) has the same
marginal law, we obtain
\begin{equation}\label{eq:Zj-mean}
  \mu_j(t)\defeq \EE{Z_j(t)}
  = \sum_{v\in\mathcal L_j} \PP{S_j(v)\le \aj}
  = 2^j\,\PP{\Gam(j,1)\le \aj},
\end{equation}
for \(j\ge 1\). (The case \(j=0\) is trivial.)

By \Cref{lem:gamma-small-ball}, for all \(j\ge 1\) such that \(\aj\in[0,1]\),
\begin{equation}\label{eq:Zj-mean-2sided}
  2^j e^{-\aj}\frac{\aj^{\,j}}{j!}
  \le
  \mu_j(t)
  \le
  2^j \frac{\aj^{\,j}}{j!}.
\end{equation}
In particular, as \(\aj\downarrow 0\),
\begin{equation}\label{eq:Zj-mean-asymp}
  \mu_j(t)
  =
  2^j \frac{\aj^{\,j}}{j!}\bigl(1+\bigO(\aj)\bigr),
\end{equation}
with the \(\bigO(\aj)\) term uniform in \(j\).

\subsection{Second moment via LCA decomposition}

Write
\[
I_v \defeq \1\{S_j(v)\le \aj\},
\]
so that
\[
Z_j(t)=\sum_{v\in\mathcal L_j} I_v
\]
and
\[
\EE{Z_j(t)^2}
=
\mu_j(t)+\sum_{v\neq w}\PP{I_v=I_w=1}.
\]

We group ordered pairs \((v,w)\) according to the depth \(\ell\in\{0,\dots,j-1\}\) of their
lowest common ancestor (LCA).
The number of such ordered pairs is
\begin{equation}\label{eq:pair-count}
N_{j,\ell}
=
2^\ell\cdot 2\cdot 2^{\,j-\ell-1}\cdot 2^{\,j-\ell-1}
=
2^{\,2j-\ell-1}.
\end{equation}
Indeed, there are \(2^\ell\) choices for the LCA vertex, a factor \(2\) for deciding which of
\((v,w)\) goes into the left or right branch below the LCA, and then \(2^{j-\ell-1}\) choices for
each descendant at depth \(j\).

Now fix an ordered pair \((v,w)\) whose LCA has depth \(\ell\).
Conditioning on the shared prefix sum \(S_\ell\sim\Gam(\ell,1)\) (with the convention
\(S_0\equiv 0\)), the two subpaths below the LCA are independent, and the additional
path-sums on the two branches both have law \(\Gam(j-\ell,1)\).
Hence
\begin{equation}\label{eq:pjl-def}
\PP{I_v=I_w=1}
=
\EE{\1_{\{S_\ell\le \aj\}}\,F_{j-\ell}(\aj-S_\ell)^2}
\defeq
p_{j,\ell},
\end{equation}
where
\[
F_m(x)\defeq \PP{\Gam(m,1)\le x},
\qquad x\ge 0,
\]
and \(F_0(x)=\1_{\{x\ge 0\}}\).

By symmetry, the quantity in \Cref{eq:pjl-def} depends only on \(j\) and \(\ell\), not on the
particular ordered pair \((v,w)\).
Therefore
\begin{equation}\label{eq:LCA-sum-exact}
\sum_{v\neq w}\PP{I_v=I_w=1}
=
\sum_{\ell=0}^{j-1} N_{j,\ell}\,p_{j,\ell}.
\end{equation}

\begin{remark}
We work with ordered pairs only for convenience; switching to unordered pairs would change
the formulas by a global factor \(2\) and would not affect the asymptotics.
The case \(\ell=0\) fits the same scheme: since \(S_0\equiv 0\), we have
\[
p_{j,0}=F_j(\aj)^2.
\]
\end{remark}

\subsection{Uniform upper bound for \texorpdfstring{$p_{j,\ell}$}{p\{j,l\}}}

Recall that for \(r\defeq j-\ell\ge 1\) we set
\[
p_{j,\ell}(a)\defeq \EE{\1_{\{S_\ell\le a\}}\,F_r(a-S_\ell)^2},
\qquad
S_\ell\sim\Gam(\ell,1),
\]
where
\[
F_m(x)\defeq \PP{\Gam(m,1)\le x},
\qquad x\ge 0.
\]

\begin{lemma}\label{lem:pjl-upper}
For all \(a>0\) and integers \(j\ge 1\), \(\ell\in\{0,\dots,j-1\}\) with \(r=j-\ell\ge 1\),
\begin{equation}\label{eq:pjl-upper}
p_{j,\ell}(a)\le \binom{2r}{r}\,\frac{a^{\,j+r}}{(j+r)!}.
\end{equation}
\end{lemma}

\begin{proof}
First assume \(\ell\ge 1\).
Using the density
\[
f_\ell(s)=\frac{s^{\ell-1}e^{-s}}{\Gamma(\ell)}
\]
of \(\Gam(\ell,1)\), we write
\[
p_{j,\ell}(a)
=
\frac{1}{\Gamma(\ell)}
\int_0^a s^{\ell-1}e^{-s}\,F_r(a-s)^2\,\dd s.
\]
Since \(e^{-s}\le 1\) and, by \Cref{lem:gamma-small-ball},
\[
F_r(x)\le \frac{x^r}{r!},
\qquad x\ge 0,
\]
we obtain
\[
p_{j,\ell}(a)
\le
\frac{1}{\Gamma(\ell)(r!)^2}
\int_0^a s^{\ell-1}(a-s)^{2r}\,\dd s.
\]
By the beta-integral identity,
\[
\int_0^a s^{\ell-1}(a-s)^{2r}\,\dd s
=
a^{\ell+2r}\,\mathrm{B}(\ell,2r+1),
\]
and hence
\[
p_{j,\ell}(a)
\le
\frac{a^{\ell+2r}}{\Gamma(\ell)(r!)^2}\,\mathrm{B}(\ell,2r+1).
\]
Using
\[
\mathrm{B}(\ell,2r+1)
=
\frac{\Gamma(\ell)\Gamma(2r+1)}{\Gamma(\ell+2r+1)}
\]
and \(\Gamma(2r+1)=(2r)!\), we get
\[
p_{j,\ell}(a)
\le
\frac{(2r)!}{(r!)^2}\,\frac{a^{\ell+2r}}{(\ell+2r)!}
=
\binom{2r}{r}\,\frac{a^{\,j+r}}{(j+r)!},
\]
since \(\ell+2r=j+r\).

For \(\ell=0\), we have \(S_0\equiv 0\), so
\[
p_{j,0}(a)=F_j(a)^2\le \left(\frac{a^j}{j!}\right)^2.
\]
Since
\[
\frac{(2j)!}{(j!)^2(2j)!}=(j!)^{-2},
\]
this agrees with \Cref{eq:pjl-upper}.
\end{proof}

\subsection{The correlation sum is negligible in the saddle window}

Set \(a\defeq \aj\), and let \(j=j(t)\) be such that \(j/t\to 1\), so that \(a\to 0\) as \(t\to\infty\).
Recall that
\[
\sum_{v\ne w}\PP{I_v=I_w=1}
=
\sum_{\ell=0}^{j-1} N_{j,\ell}\,p_{j,\ell}(a),
\qquad
N_{j,\ell}=2^{\,2j-\ell-1}.
\]

\begin{lemma}\label{lem:corr-sum-negligible}
Uniformly for \(j=j(t)\) with \(j/t\to 1\),
\begin{equation}\label{eq:corr-sum-o-mu}
\sum_{\ell=0}^{j-1} N_{j,\ell}\,p_{j,\ell}(a)
=
\smallo\!\bigl(\mu_j(t)\bigr)
\qquad (t\to\infty).
\end{equation}
\end{lemma}

\begin{proof}
By \Cref{lem:pjl-upper}, for \(r\defeq j-\ell\ge 1\),
\[
p_{j,\ell}(a)
\le
\binom{2r}{r}\,\frac{a^{\,j+r}}{(j+r)!}.
\]
Since
\[
N_{j,\ell}=2^{\,2j-\ell-1}=2^{\,j+r-1},
\]
it follows that
\[
N_{j,\ell}\,p_{j,\ell}(a)
\le
2^{\,j+r-1}\binom{2r}{r}\,\frac{a^{\,j+r}}{(j+r)!}.
\]
Summing over \(\ell\) is equivalent to summing over \(r\in\{1,\dots,j\}\), and therefore
\begin{equation}\label{eq:corr-sum-after-r}
\sum_{\ell=0}^{j-1} N_{j,\ell}\,p_{j,\ell}(a)
\le
\sum_{r=1}^{j} 2^{\,j+r-1}\binom{2r}{r}\,\frac{a^{\,j+r}}{(j+r)!}.
\end{equation}

Next, since \(a\to 0\), for all sufficiently large \(t\) we may use \Cref{eq:Zj-mean-asymp} to obtain
\[
\mu_j(t)
=
2^j\frac{a^j}{j!}\bigl(1+O(a)\bigr),
\]
and hence
\[
\frac{1}{\mu_j(t)}
\le
\frac{2\,j!}{2^j a^j}.
\]
Combining this with \Cref{eq:corr-sum-after-r}, we get
\[
\frac{1}{\mu_j(t)}
\sum_{\ell=0}^{j-1} N_{j,\ell}\,p_{j,\ell}(a)
\le
\sum_{r=1}^{j}
2^r\binom{2r}{r}a^r\,\frac{j!}{(j+r)!}.
\]
Using the elementary bounds
\[
\binom{2r}{r}\le \sum_{k=0}^{2r}\binom{2r}{k} = (1+1)^{2r}=4^r
\qquad\text{and}\qquad
(j+r)!\ge j!\,(j+1)^r,
\]
we obtain, for every \(r\ge 1\),
\[
2^r\binom{2r}{r}a^r\,\frac{j!}{(j+r)!}
\le
2^r\cdot 4^r\cdot a^r\cdot \frac{1}{(j+1)^r}
=
\left(\frac{8a}{j+1}\right)^r.
\]
Therefore
\[
\frac{1}{\mu_j(t)}
\sum_{\ell=0}^{j-1} N_{j,\ell}\,p_{j,\ell}(a)
\le
\sum_{r=1}^{j}\left(\frac{8a}{j+1}\right)^r.
\]

Set
\[
q\defeq \frac{8a}{j+1}.
\]
For all sufficiently large \(t\), we have \(q<1\), and so
\[
\sum_{r=1}^{j} q^r
=
\frac{q(1-q^j)}{1-q}
\le
\frac{q}{1-q}.
\]
Since \(j=j(t)\sim t\) and \(a=\aj\to 0\) as \(t\to\infty\), we have \(q\to 0\).
Thus
\[
\frac{1}{\mu_j(t)}
\sum_{\ell=0}^{j-1} N_{j,\ell}\,p_{j,\ell}(a)
\to 0,
\]
which proves \Cref{eq:corr-sum-o-mu}.
\end{proof}

As a consequence,
\begin{equation}\label{eq:Z2-mean}
\EE{Z_j(t)^2}
=
\mu_j(t)+o\!\bigl(\mu_j(t)\bigr)
\qquad\text{as } t\to\infty \text{ with } j=j(t)\sim t.
\end{equation}

\subsection{Hitting probability on a level}

Set \(a\defeq \aj\).
By the Cauchy--Schwarz inequality,
\[
\PP{Z_j(t)\ge 1}
\ge
\frac{\EE{Z_j(t)}^{\,2}}{\EE{Z_j(t)^2}}
=
\frac{\mu_j(t)^{\,2}}{\EE{Z_j(t)^2}}.
\]
Using \Cref{eq:Z2-mean}, we have
\[
\EE{Z_j(t)^2}
=
\mu_j(t)+\smallo\!\bigl(\mu_j(t)\bigr)
\]
uniformly for \(j=j(t)\sim t\).
Hence
\begin{equation}\label{eq:PZ1-asymp}
\PP{Z_j(t)\ge 1}
=
(1-\smallo(1))\,\mu_j(t)
\qquad (t\to\infty,\ j\sim t).
\end{equation}
Combining this with \Cref{eq:Zj-mean-asymp}, we obtain
\[
\PP{Z_j(t)\ge 1}
\sim
2^j\,\frac{a^{\,j}}{j!}
\qquad\text{for } a=\aj\to 0,\ j\sim t.
\]

\subsection{From one level to a global upper bound}

By the event inclusion \Cref{eq:suff-cond}, for every level \(j\),
\begin{equation}\label{eq:S-lb-onelevel}
\PP{\Svar>t}
\ge
\PP{Z_j(t)\ge 1}
=
(1-\smallo(1))\,\mu_j(t)
\qquad (t\to\infty,\ j\sim t),
\end{equation}
where we used \Cref{eq:PZ1-asymp}.
It therefore remains to maximize \(\mu_j(t)\) over \(j\sim t\).

Write
\[
L\defeq \log t,
\]
and parameterize
\[
j=t\Bigl(1+\frac{c}{L}\Bigr),
\qquad c>0.
\]
Then
\[
a_j(t)=\log\frac{j+1}{t}=\frac{c}{L}+\smallo(L^{-1}),
\]
and, by \Cref{eq:Zj-mean-asymp,lem:stirling},
\[
\log \mu_j(t)
=
j\log 2+j\log a_j(t)-\log j!+\smallo(t)
=
-t\log t-t\log\log t+t\bigl[\log 2+\log c+1-c\bigr]+\smallo(t).
\]
The bracketed term is maximized at \(c^\star=1\).
Thus, for
\[
j^\star
=
t\Bigl(1+\frac{1}{\log t}\Bigr)
+
\smallo\!\Bigl(\frac{t}{\log t}\Bigr),
\]
we have
\[
\log \mu_{j^\star}(t)
=
-t\log t-t\log\log t+t\log 2+\smallo(t).
\]
Combining this with \Cref{eq:S-lb-onelevel}, we obtain
\begin{equation}\label{eq:S-lb-final}
\PP{\Svar>t}
\ge
\exp\!\Bigl(-t\log t-t\log\log t+t\log 2+\smallo(t)\Bigr).
\end{equation}
Equivalently,
\[
-\log \PP{\Svar>t}
\le
t\log t+t\log\log t-t\log 2+\smallo(t).
\]

\section{Lower bound via Chernoff and a pointwise MGF bound}
\label{sec:lower}

In this section we prove the lower bound in \Cref{prop:intro-rate-bounds}.
The argument is based on the truncated sums \(\Svar^{(n)}\) introduced in \Cref{sec:tree}.
We first derive a recursive comparison inequality for their moment generating functions.
We then show that, for each fixed \(\theta>0\), the explicit value
\[
x_\theta \defeq \exp(2e^\theta)
\]
dominates this recursion.
Passing to the limit yields a pointwise bound on the moment generating function of \(\Svar\),
and a final Chernoff optimization gives the desired estimate on \(\PP{\Svar>t}\).

\subsection{Truncated sums and their recursion}
\label{subsec:trunc-recursion}

Recall from \Cref{eq:def-S} that
\[
\Svar=\sup_p \sum_{r=0}^\infty W_r(p),
\]
where the supremum is taken over all infinite rooted paths \(p\) in \(\mathcal{T}\).
We also recall from \Cref{sec:tree} the truncated sums
\[
\Svar^{(n)}=\sup_p \sum_{r=0}^{n} W_r(p),
\qquad n\ge 0,
\]
which satisfy
\[
\Svar^{(n)}\nearrow \Svar
\qquad\text{almost surely as }n\to\infty.
\]

For \(\theta> 0\), define the moment generating functions
\[
\psi_n(\theta)\defeq \EE{e^{\theta \Svar^{(n)}}},
\qquad
\psi(\theta)\defeq \EE{e^{\theta \Svar}}.
\]
By monotone convergence,
\[
\psi_n(\theta)\uparrow \psi(\theta)
\qquad\text{for each fixed }\theta> 0.
\]

\begin{lemma}[Distributional recursion for \(\Svar^{(n)}\)]
\label{lem:trunc-recursion}
For every \(n\ge 0\),
\begin{equation}\label{eq:trunc-recursion}
\Svar^{(n+1)}
\eqd
1+\max\{\U\,(\Svar^{(n)})',\,(1-\U)\,(\Svar^{(n)})''\},
\end{equation}
where \(\U\sim\Unif(0,1)\), the random variables \((\Svar^{(n)})'\) and \((\Svar^{(n)})''\) are
independent copies of \(\Svar^{(n)}\), and \(\U\) is independent of both of them.
\end{lemma}

\begin{proof}
Let \(\varnothing\) denote the root of \(\mathcal{T}\), and set
\[
\U\defeq \xi_{\varnothing}\sim\Unif(0,1).
\]
Fix \(n\ge 0\).
Every infinite rooted path starts by entering either the left or the right subtree of the root.
If a path enters the left subtree, then
\[
\sum_{r=0}^{n+1} W_r(p)
=
1+\U\sum_{r=0}^{n} W_r^{(L)}(p),
\]
where \(W_r^{(L)}(p)\) denotes the cumulative weight generated by the next \(r\) steps inside the
left subtree.
Similarly, if a path enters the right subtree, then
\[
\sum_{r=0}^{n+1} W_r(p)
=
1+(1-\U)\sum_{r=0}^{n} W_r^{(R)}(p),
\]
where \(W_r^{(R)}(p)\) is defined analogously inside the right subtree.

Taking the supremum over all paths yields the almost sure identity
\[
\Svar^{(n+1)}
=
1+\max\left\{
\U\,\sup_{p\text{ in left subtree}}\sum_{r=0}^{n} W_r^{(L)}(p),\,
(1-\U)\,\sup_{p\text{ in right subtree}}\sum_{r=0}^{n} W_r^{(R)}(p)
\right\}.
\]
By the i.i.d.\ labelling of the vertices, the two subtree suprema are independent, each has the
same law as \(\Svar^{(n)}\), and both are independent of \(\U\).
This proves \Cref{eq:trunc-recursion}.
\end{proof}

\subsection{MGF comparison at a fixed parameter}
\label{subsec:mgf-recursion}

We collect here the technical ingredients needed for the Chernoff bound.
H\"older's inequality yields a log-convexity estimate for \(\psi_n(\theta u)\), which in turn gives
a recursive upper bound
\[
\psi_{n+1}(\theta)\le G_\theta(\psi_n(\theta))
\]
for a suitable scalar map \(G_\theta\).
We then show that any fixed point upper barrier \(x_\theta\ge 1\) with
\[
G_\theta(x_\theta)\le x_\theta
\]
implies \(\psi(\theta)\le x_\theta\).
The next subsection provides an explicit choice of \(x_\theta\) and performs the final Chernoff
optimization.

\begin{lemma}[H\"older / log-convexity]
\label{lem:holder-logconvex}
For every \(n\ge 0\), every \(\theta> 0\), and every \(u\in[0,1]\),
\begin{equation}\label{eq:psi-logconvex}
\psi_n(\theta u)\le \psi_n(\theta)^u.
\end{equation}
\end{lemma}

\begin{proof}
By H\"older's inequality with exponents \(p=1/u\) and \(q=1/(1-u)\) (interpreting the cases
\(u=0\) and \(u=1\) by continuity),
\[
\psi_n(\theta u)
=
\EE{\bigl(e^{\theta \Svar^{(n)}}\bigr)^u}
\le
\bigl(\EE{e^{\theta \Svar^{(n)}}}\bigr)^u
=
\psi_n(\theta)^u.
\]
\end{proof}

For each fixed \(\theta> 0\), define the scalar map
\begin{equation}\label{eq:Gtheta-def}
G_\theta(x)\defeq 2e^\theta\,\frac{x-1}{\log x}
\qquad (x>1),
\qquad
G_\theta(1)\defeq \lim_{x\downarrow 1}G_\theta(x)=2e^\theta.
\end{equation}
Since \(x\mapsto (x-1)/\log x\) is increasing on \([1,\infty)\), the map \(G_\theta\) is continuous
and increasing on \([1,\infty)\).

\begin{lemma}[Recursive bound for the MGF]
\label{lem:local-onestep}
For every \(n\ge 0\) and every \(\theta> 0\),
\begin{equation}\label{eq:local-onestep}
\psi_{n+1}(\theta)\le G_\theta\!\bigl(\psi_n(\theta)\bigr).
\end{equation}
\end{lemma}

\begin{proof}
By the distributional recursion in \Cref{eq:trunc-recursion} and the elementary inequality
\(e^{\max\{x,y\}}\le e^x+e^y\),
\[
\psi_{n+1}(\theta)
=
\EE{e^{\theta \Svar^{(n+1)}}}
\le
e^\theta\,\EE{e^{\theta \U (\Svar^{(n)})'}+e^{\theta (1-\U)(\Svar^{(n)})''}}.
\]
Conditioning on \(\U\sim \Unif(0,1)\) and using independence of \(\U\), \((\Svar^{(n)})'\), and
\((\Svar^{(n)})''\), we obtain
\[
\psi_{n+1}(\theta)
\le
2e^\theta\int_0^1 \EE{e^{\theta u \Svar^{(n)}}}\,\dd u
=
2e^\theta\int_0^1 \psi_n(\theta u)\,\dd u.
\]
By \Cref{lem:holder-logconvex},
\[
\int_0^1 \psi_n(\theta u)\,\dd u
\le
\int_0^1 \psi_n(\theta)^u\,\dd u
=
\frac{\psi_n(\theta)-1}{\log \psi_n(\theta)},
\]
where we use the identity
\[
\int_0^1 a^u\,\dd u=\frac{a-1}{\log a},
\qquad a>0,
\]
with the value at \(a=1\) understood by continuity.
Substituting this into the previous estimate gives \Cref{eq:local-onestep}.
\end{proof}

\begin{lemma}[Scalar domination at a fixed parameter]
\label{lem:scalar-domination}
Fix \(\theta>0\), and let \(x_\theta\ge 1\) satisfy
\begin{equation}\label{eq:scalar-barrier}
G_\theta(x_\theta)\le x_\theta.
\end{equation}
Then
\[
\psi(\theta)\le x_\theta.
\]
\end{lemma}

\begin{proof}
Set \(a_n\defeq \psi_n(\theta)\) for \(n\ge 0\).
Since
\[
\frac{x-1}{\log x}\ge 1
\qquad (x\ge 1),
\]
we have
\[
G_\theta(x)\ge 2e^\theta
\qquad (x\ge 1).
\]
Thus \eqref{eq:scalar-barrier} implies
\[
x_\theta\ge G_\theta(x_\theta)\ge 2e^\theta.
\]
In particular,
\[
a_0=\psi_0(\theta)=e^\theta\le x_\theta.
\]

Moreover, by \Cref{eq:local-onestep},
\[
a_{n+1}\le G_\theta(a_n)
\qquad (n\ge 0).
\]
Since \(G_\theta\) is increasing and \(G_\theta(x_\theta)\le x_\theta\), induction yields
\[
a_n\le x_\theta
\qquad\text{for all }n\ge 0.
\]
Finally, \(a_n=\psi_n(\theta)\uparrow \psi(\theta)\), so passing to the limit gives
\[
\psi(\theta)\le x_\theta.
\]
\end{proof}

\subsection{An explicit scalar barrier and Chernoff optimization}
\label{subsec:chernoff-final}

For each fixed \(\theta>0\), set
\[
x_\theta \defeq \exp(2e^\theta).
\]

\begin{proposition}\label{prop:mgf-pointwise}
For every \(\theta>0\),
\begin{equation}\label{eq:psi-pointwise}
\psi(\theta)\le \exp(2e^\theta).
\end{equation}
\end{proposition}

\begin{proof}
Let \(\theta>0\) be fixed.
Since
\[
\log x_\theta = 2e^\theta,
\]
the definition of \(G_\theta\) gives
\[
G_\theta(x_\theta)
=
2e^\theta\,\frac{x_\theta-1}{\log x_\theta}
=
2e^\theta\,\frac{x_\theta-1}{2e^\theta}
=
x_\theta-1
\le x_\theta.
\]
Thus \Cref{lem:scalar-domination} applies and yields
\[
\psi(\theta)\le x_\theta=\exp(2e^\theta).
\]
\end{proof}

\begin{corollary}\label{cor:chernoff-pointwise}
For every \(t>0\),
\begin{equation}\label{eq:chernoff-inf}
\PP{\Svar>t}
\le
\inf_{\theta>0}\exp\!\bigl(-\theta t+2e^\theta\bigr).
\end{equation}
In particular, for every \(t>2\),
\begin{equation}\label{eq:chernoff-explicit}
\PP{\Svar>t}
\le
\exp\!\Bigl(-t\log(t/2)+t\Bigr)
=
\exp\!\Bigl(-t\log t+(1+\log 2)t\Bigr).
\end{equation}
Equivalently,
\begin{equation}\label{eq:lower-bound-final}
-\log \PP{\Svar>t}
\ge
t\log t-(1+\log 2)t,
\qquad t>2.
\end{equation}
\end{corollary}

\begin{proof}
By Chernoff's inequality and \Cref{eq:psi-pointwise}, for every \(\theta>0\),
\[
\PP{\Svar>t}
\le
e^{-\theta t}\psi(\theta)
\le
\exp\!\bigl(-\theta t+2e^\theta\bigr).
\]
Taking the infimum over \(\theta>0\) proves \Cref{eq:chernoff-inf}.

Now fix \(t>2\), and consider
\[
f_t(\theta)\defeq -\theta t+2e^\theta,
\qquad \theta>0.
\]
Then
\[
f_t'(\theta)=-t+2e^\theta,
\qquad
f_t''(\theta)=2e^\theta>0,
\]
so \(f_t\) is strictly convex.
Its unique critical point is
\[
\theta^\star=\log(t/2)>0,
\]
and therefore \(\theta^\star\) is the global minimizer of \(f_t\) on \((0,\infty)\).
Substituting \(\theta^\star\) into \Cref{eq:chernoff-inf} gives
\[
\PP{\Svar>t}
\le
\exp\!\bigl(f_t(\theta^\star)\bigr)
=
\exp\!\Bigl(-t\log(t/2)+t\Bigr),
\]
which is \Cref{eq:chernoff-explicit}.
Taking minus logarithms yields \Cref{eq:lower-bound-final}.
\end{proof}

\section{A monotone distribution--function scheme for bounding \texorpdfstring{$\EE{\Svar}$}{E[S]}}
\label{sec:mean}

In this section we develop a monotone scheme for obtaining explicit upper bounds on
\(\EE{\Svar}\).
The scheme combines the fixed--point equation for \(\Svar\) with the explicit tail majorant
for \(\PP{\Svar>t}\) derived in \Cref{sec:lower}, and is naturally suited to a computer--assisted
implementation.

Define
\[
F(x)\defeq \PP{\Svar\le x}, \qquad x\in\R.
\]
Since \(\supp(\Svar)\subseteq [2,\infty)\), the distribution function \(F\) vanishes on
\((-\infty,2)\) and is nondecreasing on \(\R\).
Furthermore,
\[
\Svar \eqd 1+\max\{\U \Svar', (1-\U)\Svar''\},
\]
where \(\U\sim \Unif(0,1)\), and \(\Svar',\Svar''\) are independent copies of \(\Svar\),
independent of \(\U\).

\begin{proposition}[Distribution--function fixed--point equation]
\label{prop:distribution-function-equation}
For every \(x\in\R\),
\begin{equation}\label{eq:F-fixed-point}
F(x)=\int_0^1
F\!\left(\frac{x-1}{u}\right)
F\!\left(\frac{x-1}{1-u}\right)\dd u.
\end{equation}
\end{proposition}

\begin{proof}
By the distributional fixed--point equation for \(\Svar\),
\[
\Svar \eqd 1+\max\{\U \Svar',(1-\U)\Svar''\},
\]
where \(\U\sim \Unif(0,1)\), and \(\Svar',\Svar''\) are independent copies of \(\Svar\),
independent of \(\U\). Hence, conditioning on \(\U\),
\[
F(x)
=
\PP{1+\max\{\U \Svar',(1-\U)\Svar''\}\le x}
=
\EE{\PP{\U \Svar'\le x-1,\ (1-\U)\Svar''\le x-1 \mid \U}}.
\]
Since \(\Svar'\) and \(\Svar''\) are independent and both have distribution function \(F\), this gives
\[
F(x)
=
\int_0^1
\PP{\Svar'\le (x-1)/u}\,
\PP{\Svar''\le (x-1)/(1-u)}\dd u
=
\int_0^1
F\!\left(\frac{x-1}{u}\right)
F\!\left(\frac{x-1}{1-u}\right)\dd u.
\]
This proves \Cref{eq:F-fixed-point}.
\end{proof}

\begin{proposition}[A general upper bound on \(\EE{\Svar}\)]
\label{prop:mean-upper-general}
Let \(A>2\), and let \(L:\R\to[0,1]\) satisfy
\[
L(x)\le F(x),
\qquad x\in(2,A).
\]
Then
\[
\EE{\Svar}
\le
2+\int_2^A \bigl(1-L(x)\bigr)\,\dd x
+\frac{\exp\!\bigl(-A\log(A/2)+A\bigr)}{\log(A/2)}.
\]
\end{proposition}

\begin{proof}
We split the tail-integral representation at \(A\):
\[
\EE{\Svar}
=
2+\int_2^A \bigl(1-F(x)\bigr)\,\dd x
+\int_A^\infty \bigl(1-F(x)\bigr)\,\dd x.
\]
Since \(L\le F\) on \((2,A)\), we have
\[
\int_2^A \bigl(1-F(x)\bigr)\,\dd x
\le
\int_2^A \bigl(1-L(x)\bigr)\,\dd x.
\]
Moreover, by \Cref{cor:chernoff-pointwise},
\[
\int_A^\infty \bigl(1-F(x)\bigr)\,\dd x
=
\int_A^\infty \PP{\Svar>x}\,\dd x
\le
\int_A^\infty \exp\!\bigl(-x\log(x/2)+x\bigr)\,\dd x.
\]
Applying \Cref{lem:tail-integral-bound} to the last integral gives the claim.
\end{proof}

\subsection{The nonlinear operator and its basic properties}

It is convenient to express the distribution--function fixed--point equation in operator form.
For every measurable function \(G:\R\to[0,1]\), define
\[
(\mathcal K G)(x)\defeq
\int_0^1
G\!\left(\frac{x-1}{u}\right)
G\!\left(\frac{x-1}{1-u}\right)\dd u,
\qquad x\in\R.
\]
Then \Cref{eq:F-fixed-point} may be written simply as
\[
F=\mathcal K F
\qquad\text{on }\R.
\]

Note that if \(G:\R\to[0, 1]\) is measurable, then so is \(\mathcal K G\).
Indeed, \(\mathcal K G\) is measurable and satisfies
\[
0\le (\mathcal K G)(x)\le 1,
\qquad x\in\R.
\]
Hence, for any measurable function \(L_0:\R\to[0, 1]\), we may define
recursively
\[
L_{n+1}\defeq \mathcal K L_n,
\qquad n\ge 0.
\]

\begin{lemma}[Order preservation and monotonicity of \(\mathcal K\)]
\label{lem:K-monotone}
Let \(G,H:\R\to[0,1]\) be measurable functions.

\begin{enumerate}
\item If \(G\le H\) pointwise on \(\R\), then
\[
\mathcal K G\le \mathcal K H
\qquad\text{pointwise on }\R.
\]

\item If \(G\) is nondecreasing on \(\R\), then \(\mathcal K G\) is also nondecreasing on \(\R\).
\end{enumerate}
\end{lemma}

\begin{proof}
Assume first that \(G\le H\) pointwise.
Then for every \(x\in\R\) and every \(u\in(0,1)\),
\[
G\!\left(\frac{x-1}{u}\right)
G\!\left(\frac{x-1}{1-u}\right)
\le
H\!\left(\frac{x-1}{u}\right)
H\!\left(\frac{x-1}{1-u}\right),
\]
since all factors are nonnegative.
Integrating over \(u\in(0,1)\) yields
\[
(\mathcal K G)(x)\le (\mathcal K H)(x),
\qquad x\in\R.
\]

Now suppose that \(G\) is nondecreasing, and let \(x\le y\).
For every \(u\in(0,1)\),
\[
\frac{x-1}{u}\le \frac{y-1}{u},
\qquad
\frac{x-1}{1-u}\le \frac{y-1}{1-u}.
\]
Hence, by monotonicity of \(G\),
\[
G\!\left(\frac{x-1}{u}\right)
G\!\left(\frac{x-1}{1-u}\right)
\le
G\!\left(\frac{y-1}{u}\right)
G\!\left(\frac{y-1}{1-u}\right).
\]
Integrating over \(u\in(0,1)\) gives
\[
(\mathcal K G)(x)\le (\mathcal K G)(y).
\]
Thus \(\mathcal K G\) is nondecreasing.
\end{proof}

\begin{corollary}
\label{cor:L_n_iteration_properties}
Assume that \(L_0:\R\to[0,1]\) is measurable, and let
\[
L_{n+1}\defeq \mathcal K L_n,
\qquad n\ge 0.
\]

\begin{enumerate}
\item If \(L_0\) is nondecreasing, then every \(L_n\) is nondecreasing.

\item If \(L_0\le F\) on \(\R\), then
\[
L_n\le F
\qquad\text{on }\R
\]
for every \(n\ge 0\).

\item If \(L_0\le L_1\) on \(\R\), then
\[
L_n\le L_{n+1}
\qquad\text{on }\R
\]
for every \(n\ge 0\).
\end{enumerate}
\end{corollary}

\begin{proof}
Each claim follows by induction on \(n\).

For (1), if \(L_n\) is nondecreasing, then \(\mathcal K L_n=L_{n+1}\) is nondecreasing by
\Cref{lem:K-monotone}.

For (2), if \(L_n\le F\), then the order-preserving property of \(\mathcal K\) and the identity
\(F=\mathcal K F\) give
\[
L_{n+1}=\mathcal K L_n\le \mathcal K F = F.
\]

For (3), if \(L_n\le L_{n+1}\), then again by order preservation,
\[
L_{n+1}=\mathcal K L_n\le \mathcal K L_{n+1}=L_{n+2}.
\]
This proves the result.
\end{proof}

\subsection{A conservative mesh scheme}
\label{subsec:mesh-scheme}

To turn \Cref{prop:mean-upper-general} into an explicit numerical bound, we replace the
operator \(\mathcal K\) by a conservative lower approximation on \([2,A)\). If
\(L:\R\to[0,1]\) is nondecreasing, then for every \(x\in\R\) and every \(0\leq a < b \leq 1\),
\[
\int_a^b
L\!\left(\frac{x-1}{u}\right)
L\!\left(\frac{x-1}{1-u}\right)\dd u
\ge
(b-a)\,
L\!\left(\frac{x-1}{b}\right)
L\!\left(\frac{x-1}{1-a}\right).
\]
This leads to the following lower Riemann-type approximation of \((\mathcal K L)(x)\).

Fix meshes
\[
0=u_0<u_1<\cdots<u_M=1
\qquad\text{and}\qquad
2=x_0<x_1<\cdots<x_N=A.
\]
Let \(L:\R\to[0,1]\) be nondecreasing.

For \(k\in\{0,\dots,N-1\}\), define
\begin{equation}\label{eq:Q-definition}
(\mathcal QL)(x_k)
\defeq
\sum_{r=0}^{M-1}
(u_{r+1}-u_r)\,
L\!\left(\frac{x_k-1}{u_{r+1}}\right)
L\!\left(\frac{x_k-1}{1-u_r}\right).
\end{equation}

We then define the lower mesh operator \(\underline{\mathcal K}\) by
\begin{equation}\label{eq:K-lower-definition}
(\underline{\mathcal K}L)(x)\defeq
\begin{cases}
0, & x<2,\\[1mm]
(\mathcal QL)(x_k), & x\in[x_k,x_{k+1}),\quad 0\le k\le N-1,\\[1mm]
\max\!\bigl\{(\mathcal QL)(x_{N-1}),\,L_0(x)\bigr\}, & x\ge A.
\end{cases}
\end{equation}

\begin{proposition}[Conservative discretization of \(\mathcal K\)]
\label{prop:mesh-lower-bound}
Let \(L:\R\to[0,1]\) be nondecreasing. Then \(\underline{\mathcal K}L\) is nondecreasing, and
\[
(\underline{\mathcal K}L)(x)\le (\mathcal K L)(x),
\qquad x\in[2,A).
\]
\end{proposition}

\begin{proof}
Fix \(k\in\{0,\dots,N-1\}\). For \(u\in[u_r,u_{r+1}]\), monotonicity of \(L\) yields
\[
L\!\left(\frac{x_k-1}{u}\right)
L\!\left(\frac{x_k-1}{1-u}\right)
\ge
L\!\left(\frac{x_k-1}{u_{r+1}}\right)
L\!\left(\frac{x_k-1}{1-u_r}\right).
\]
Integrating over \([u_r,u_{r+1}]\) and summing in \(r\) gives
\[
(\mathcal K L)(x_k)\ge (\mathcal QL)(x_k).
\]
Since \(\mathcal K L\) is nondecreasing by \Cref{lem:K-monotone}, for every
\(x\in[x_k,x_{k+1})\),
\[
(\mathcal K L)(x)\ge (\mathcal K L)(x_k)\ge (\mathcal QL)(x_k)
=(\underline{\mathcal K}L)(x).
\]
This proves
\[
(\underline{\mathcal K}L)(x)\le (\mathcal K L)(x),
\qquad x\in[2,A).
\]

For monotonicity, each summand in \((\mathcal QL)(x_k)\) is nondecreasing in \(x_k\), so
\(k\mapsto (\mathcal QL)(x_k)\) is nondecreasing. Hence \(\underline{\mathcal K}L\) is
nondecreasing on \([2,A)\). The definition on \((-\infty,2)\) and \([A,\infty)\) preserves
monotonicity, so \(\underline{\mathcal K}L\) is nondecreasing on \(\R\).
\end{proof}

Starting from a nondecreasing function \(L_0:\R\to[0,1]\) such that
\[
L_0(x)\le F(x),
\qquad x\in\R,
\]
define recursively
\begin{equation}\label{eq:discrete-iteration}
L_{m+1}\defeq \underline{\mathcal K}\,L_m,
\qquad m\ge 0.
\end{equation}

\begin{corollary}
\label{cor:discrete-iteration-below-F}
For every \(m\ge 0\), the function \(L_m\) is nondecreasing and satisfies
\[
L_m(x)\le F(x),
\qquad x\in\R.
\]
\end{corollary}

\begin{proof}
We argue by induction on \(m\). The case \(m=0\) is immediate.

Assume that \(L_m\) is nondecreasing and that \(L_m\le F\) on \(\R\).
Then \(L_{m+1}=\underline{\mathcal K}\,L_m\) is nondecreasing by
\Cref{prop:mesh-lower-bound}. It remains to show that \(L_{m+1}\le F\).

If \(x\in[2,A)\), then by \Cref{prop:mesh-lower-bound},
\[
L_{m+1}(x)
=
(\underline{\mathcal K}\,L_m)(x)
\le
(\mathcal K L_m)(x),
\]
and since \(L_m\le F\), the order-preserving property from \Cref{lem:K-monotone} gives
\[
(\mathcal K L_m)(x)\le (\mathcal K F)(x)=F(x).
\]

If \(x<2\), then \(L_{m+1}(x)=0\le F(x)\).

Finally, if \(x\ge A\), then by definition
\[
L_{m+1}(x)
=
\max\!\bigl\{(\mathcal Q L_m)(x_{N-1}),\,L_0(x)\bigr\}.
\]
Since \(x_{N-1}\in[2,A)\), the first part of the argument applied at \(x=x_{N-1}\) yields
\[
(\mathcal Q L_m)(x_{N-1})
=
L_{m+1}(x_{N-1})
\le
F(x_{N-1})
\le
F(x),
\]
while \(L_0(x)\le F(x)\) by assumption. Hence \(L_{m+1}(x)\le F(x)\).

This completes the induction.
\end{proof}

\subsection{Algorithmic realization of the mesh scheme}

We now describe the finite-dimensional realization of the conservative mesh iteration from
\Cref{subsec:mesh-scheme}. Fix a nondecreasing function \(L_0:\R\to[0,1]\) such that
\(L_0\le F\), and fix meshes
\[
0=u_0<u_1<\cdots<u_M=1
\qquad\text{and}\qquad
2=x_0<x_1<\cdots<x_N=A.
\]
Then each iterate \(L_m\) is determined by its values on the grid
\(\{x_0,\dots,x_{N-1}\}\), since
\[
L_m(x)=
\begin{cases}
0, & x<2,\\[1mm]
L_m(x_k), & x\in[x_k,x_{k+1}),\quad 0\le k\le N-1,\\[1mm]
\max\!\bigl\{L_m(x_{N-1}),\,L_0(x)\bigr\}, & x\ge A.
\end{cases}
\]
Accordingly, we represent \(L_m\) on \([2,A)\) by the vector
\[
v^{(m)}\defeq \bigl(L_m(x_0),\dots,L_m(x_{N-1})\bigr)\in[0,1]^N.
\]

The next iterate is then obtained from
\[
v_k^{(m+1)}
=
\sum_{r=0}^{M-1}
(u_{r+1}-u_r)\,
L_m\!\left(\frac{x_k-1}{u_{r+1}}\right)
L_m\!\left(\frac{x_k-1}{1-u_r}\right),
\qquad 0\le k\le N-1.
\]
Thus, to pass from step \(m\) to step \(m+1\), it suffices to know the vector \(v^{(m)}\) and to
evaluate \(L_0\) at the finitely many points
\[
\frac{x_k-1}{u_{r+1}},
\qquad
\frac{x_k-1}{1-u_r},
\qquad
0\le k\le N-1,\quad 0\le r\le M-1.
\]

For each \(m\ge 1\), the corresponding upper bound on \(\EE{\Svar}\) is
\[
U_m
\defeq
2+
\sum_{k=0}^{N-1}(x_{k+1}-x_k)\bigl(1-v_k^{(m)}\bigr)
+
\frac{\exp\!\bigl(-A\log(A/2)+A\bigr)}{\log(A/2)}.
\]
By \Cref{cor:discrete-iteration-below-F}, every iterate produced in this way remains a
nondecreasing lower bound for \(F\), and hence each \(U_m\) is a rigorous upper bound on
\(\EE{\Svar}\).

\subsection{A floating-point pilot computation}
\label{subsec:floating-point-pilot}

We conclude with a floating-point illustration of the conservative mesh scheme from the
previous subsection. This computation is included only as a numerical experiment and not as a
certified proof.

We take as input the function
\[
L_0(x)=
\begin{cases}
0, & x<2,\\[1mm]
\max\!\bigl\{0,\,1-\exp\!\bigl(-x\log(x/2)+x\bigr)\bigr\}, & x\ge 2.
\end{cases}
\]
By \Cref{cor:chernoff-pointwise}, this function satisfies \(L_0\le F\) on \(\R\).

For the basic pilot computation, we choose
\[
A=10,
\qquad
N=M=100,
\]
with uniform meshes
\[
2=x_0<x_1<\cdots<x_N=A,
\qquad
0=u_0<u_1<\cdots<u_M=1.
\]
Using the finite-dimensional scheme from the previous subsection, we compute \(50\) mesh
iterates \(L_m\). The corresponding upper bounds \(U_m\) decrease rapidly during the first
several iterations and then appear to stabilize near \(4.34\); see
\Cref{fig:mean-bound-plot}. The associated mesh iterates are shown in
\Cref{fig:iterates-plot}, where the curve labelled \(m=0\) is the step representation of the
input function \(L_0\).

We also include a higher-resolution experiment with the same value \(A=10\), but with
\[
N=M=10^4,
\qquad
m_{\max}=50,
\]
shown in \Cref{fig:iterates-plot-2}. In this case, the upper bound after \(50\) iterations is
approximately \(4.09\).

The computation was carried out in standard floating-point arithmetic in Google Colab.
Accordingly, the resulting figures should be interpreted only as a pilot experiment and not as a
certified numerical proof. A fully rigorous implementation would require interval arithmetic
together with explicit control of discretization and rounding errors.

\begin{figure}[htbp]
\centering
\includegraphics[width=0.8\textwidth]{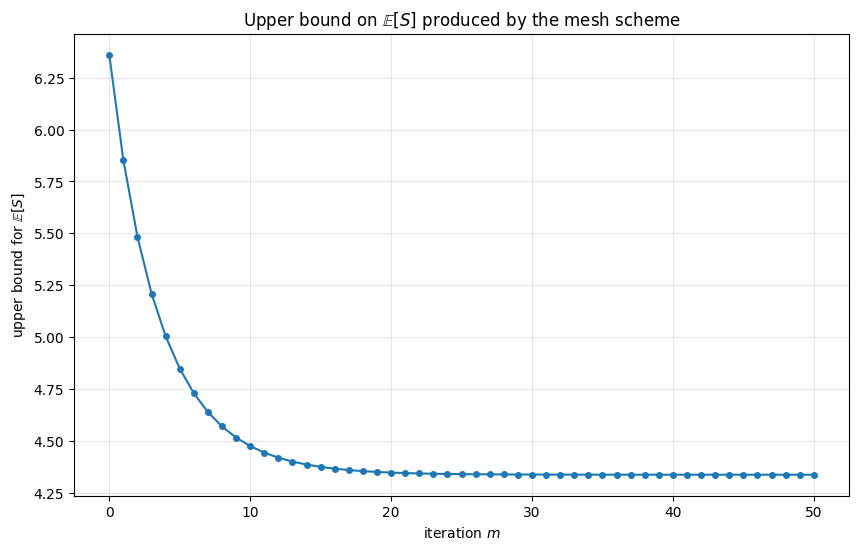}
\caption{Upper bounds \(U_m\) produced by the mesh scheme as a function of the iteration
number.}
\label{fig:mean-bound-plot}
\end{figure}

\begin{figure}[H]
\centering

\begin{subfigure}{\textwidth}
    \centering
    \includegraphics[width=1.0\textwidth]{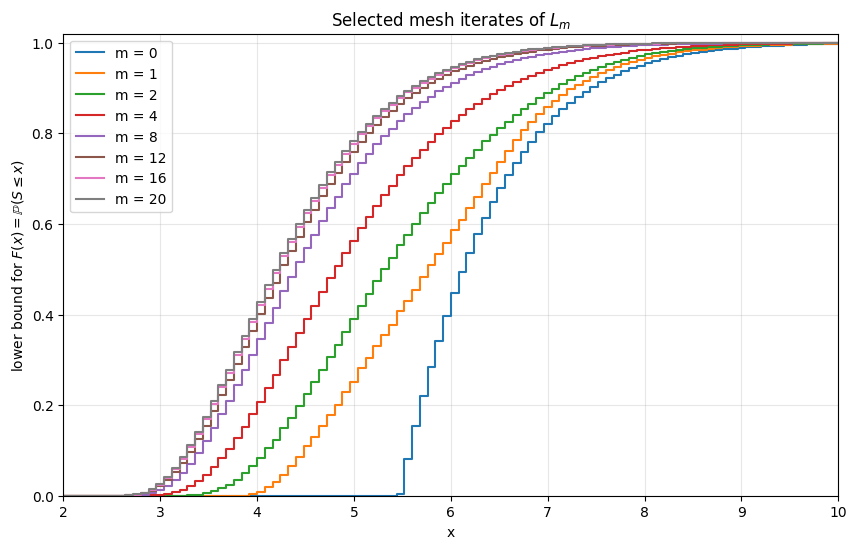}
    \caption{Selected mesh iterates on the interval \([2,A]\). The curve labelled \(m=0\) is the step
    representation of the input function \(L_0\); the remaining curves are produced by the
    conservative mesh iteration. In the present pilot computation, the profiles move upward toward
    the distribution function \(F(x)=\PP{\Svar\le x}\).}
    \label{fig:iterates-plot}
\end{subfigure}

\medskip

\begin{subfigure}{\textwidth}
    \centering
    \includegraphics[width=1.0\textwidth]{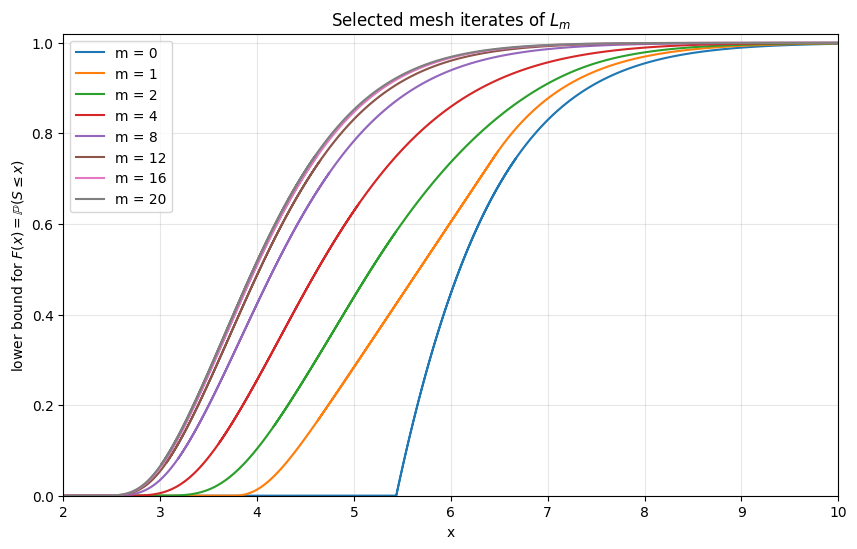}
    \caption{Selected mesh iterates of \(L_m\) on the interval \([2,A]\) for the floating-point experiment with
    \(A=10\), \(N=M=10^4\), and \(m_{\max}=50\).}
    \label{fig:iterates-plot-2}
\end{subfigure}

\caption{Selected mesh iterates produced by the conservative mesh scheme for two floating-point experiments.}
\label{fig:iterates-comparison}
\end{figure}

\newpage
\appendix
\crefalias{section}{appendix}
\section{Auxiliary estimates}\label{app:aux}

We record two standard estimates that will be used repeatedly in the proof of the tail bounds:
a uniform small-ball bound for Gamma random variables and Stirling's formula.
Throughout, \(\Gam(m,1)\) denotes a Gamma random variable with shape parameter \(m \ge 1\) and rate \(1\).

\begin{lemma}[Uniform Gamma small-ball bound]
\label{lem:gamma-small-ball}
For every integer \(m \ge 1\) and every \(x \in [0,1]\),
\[
e^{-x}\frac{x^m}{m!}
\;\le\;
\PP{\Gam(m,1)\le x}
\;\le\;
\frac{x^m}{m!}.
\]
Consequently, uniformly for \(m \ge 1\) and \(x \in [0,1]\),
\[
\PP{\Gam(m,1)\le x}
=
\frac{x^m}{m!}\bigl(1+\bigO(x)\bigr).
\]
\end{lemma}

\begin{proof}
By definition,
\[
\PP{\Gam(m,1)\le x}
=
\frac{1}{\Gamma(m)}
\int_0^x s^{m-1}e^{-s}\,\dd s
=
\frac{1}{(m-1)!}
\int_0^x s^{m-1}e^{-s}\,\dd s.
\]
For \(s \in [0,x]\) we have \(e^{-x} \le e^{-s} \le 1\), and therefore
\[
\frac{e^{-x}}{(m-1)!}\int_0^x s^{m-1}\,\dd s
\;\le\;
\PP{\Gam(m,1)\le x}
\;\le\;
\frac{1}{(m-1)!}\int_0^x s^{m-1}\,\dd s.
\]
Evaluating the integral gives
\[
e^{-x}\frac{x^m}{m!}
\;\le\;
\PP{\Gam(m,1)\le x}
\;\le\;
\frac{x^m}{m!},
\]
which proves the first claim.
Since \(e^{-x}=1+\bigO(x)\) uniformly for \(x\in[0,1]\), the second estimate follows immediately.
\end{proof}

\begin{lemma}[Stirling expansion]
\label{lem:stirling}
Uniformly for integers \(n \ge 1\),
\[
\log n!
=
n\log n - n + \frac12 \log(2\pi n) + \bigO(1/n).
\]
\end{lemma}

\begin{proof}
This is the classical Stirling formula.
\end{proof}

\begin{lemma}[A simple tail bound for the explicit majorant]
\label{lem:tail-integral-bound}
For every \(A>2\),
\[
\int_A^\infty \exp\!\bigl(-x\log(x/2)+x\bigr)\,\dd x
\le
\frac{\exp\!\bigl(-A\log(A/2)+A\bigr)}{\log(A/2)}.
\]
\end{lemma}

\begin{proof}
Define
\[
g(x)\defeq x\log(x/2)-x,
\qquad x>0.
\]
Then
\[
g'(x)=\log(x/2).
\]
Since \(A>2\), we have \(g'(A)=\log(A/2)>0\). Moreover, for every \(x\ge A\),
the mean value theorem gives
\[
g(x)-g(A)=g'(\xi)(x-A)
\]
for some \(\xi\in[A,x]\). Because \(g'(x)=\log(x/2)\) is increasing, we have
\(g'(\xi)\ge g'(A)=\log(A/2)\), and therefore
\[
g(x)\ge g(A)+\log(A/2)\,(x-A),
\qquad x\ge A.
\]
Exponentiating the negative of both sides yields
\[
\exp\!\bigl(-g(x)\bigr)
\le
\exp\!\bigl(-g(A)\bigr)\exp\!\bigl(-\log(A/2)(x-A)\bigr).
\]
Hence
\[
\int_A^\infty \exp\!\bigl(-g(x)\bigr)\,\dd x
\le
\exp\!\bigl(-g(A)\bigr)
\int_A^\infty \exp\!\bigl(-\log(A/2)(x-A)\bigr)\,\dd x.
\]
After the change of variables \(u=x-A\), this becomes
\[
\int_A^\infty \exp\!\bigl(-g(x)\bigr)\,\dd x
\le
\exp\!\bigl(-g(A)\bigr)\int_0^\infty e^{-\log(A/2)u}\,\dd u
=
\frac{\exp\!\bigl(-g(A)\bigr)}{\log(A/2)}.
\]
Substituting \(g(A)=A\log(A/2)-A\) proves the claim.
\end{proof}

\clearpage
\printbibliography

\end{document}